\documentclass[12pt]{article}
\usepackage{amssymb,amsmath,amsthm,amscd,tikz}

\title{Closed Aspherical Manifolds with Center \\ {\small (to Frank Raymond, with friendship and admiration upon his 80th birthday)}}
\author{Sylvain Cappell\thanks{Research was partially supported by a DARPA grant} \\
Courant Institute, New York University \\
\\
Shmuel Weinberger\thanks{Research was partially supported by an NSF grant} \\
University of Chicago \\
\\
Min Yan\thanks{Research was supported by Hong Kong Research Grant Council General Research Fund 604408 and 605610} \\ 
Hong Kong University of Science and Technology}

\newcommand{\sub}{\subset}
\newcommand{\pa}{\partial}

\newcommand{\bb}{\mathbb}

\newtheorem{theorem}{Theorem}

\newtheorem*{theorem*}{Theorem}
\newtheorem*{conclusion*}{Conclusion}
\newtheorem*{proposition*}{Proposition}
\newtheorem*{corollary*}{Corollary}
\newtheorem*{definition*}{Definition}
\newtheorem*{lemma*}{Lemma}

\theoremstyle{definition}

\newtheorem{remark}[theorem]{Remark}

\newtheorem*{remark*}{Remark}

\begin{document}
\maketitle

\begin{abstract}
We show that in all dimensions $> 7$ there are closed aspherical manifolds whose fundamental groups have nontrivial center but do not possess any topological circle actions.  This disproves a conjectured converse (proposed by Conner and Raymond) to a classical theorem of Borel.
\end{abstract}

\section{Introduction}

Aspherical manifolds are manifolds with contractible universal cover. Such manifolds play significant roles in mathematics. Their homotopy types are determined by the fundamental groups. Important rigidity conjectures state more strongly that the geometry of such spaces is specified by their fundamental group. For example, the Borel conjecture states that the homeomorphism type of a closed aspherical manifold is determined by the fundamental group. Progress on such rigidity conjectures has long been a focus of research efforts.

Here we consider another kind of conjectured rigidity of aspherical manifolds, about the existence of effective actions of positive dimensional compact Lie groups. A theorem of Borel \cite{borel} asserts that, if the fundamental group of a closed aspherical manifold is centerless, then there is no such action, and the finite groups that act effectively must act as outer automorphisms of the fundamental group. In fact, Borel's analysis shows that for a circle action on a closed aspherical manifold, the orbit of any point is central in the fundamental group, and the inclusion map is an injection on fundamental group. Then Conner and Raymond proved in \cite{cr1} that only toral groups, among the connected Lie groups, can act effectively on closed aspherical manifolds, and the centers of their fundamental groups must contain an injective image of the fundamental group of the torus. They further raised the question \cite[page 229]{cr1} whether a converse to Borel's theorem holds: If the fundamental group of a closed aspherical manifold has nontrivial center, then the manifold has a circle action, such that the orbit circle is a nontrivial central element of the fundamental group. This conjecture was stated again recently in \cite{lr1}. By the works of \cite{cj, gabai}, the conjecture is true in dimension $3$.

In the decades since this speculation arose, many of the conjectures about aspherical manifolds from that period were disproved. Essentially, the only remaining one is the Borel conjecture that asserts a form of topological rigidity for aspherical manifolds; see \cite{bl, farrell} for some of the more recent positive results on these problems.

In a sense, the Conner-Raymond conjecture can be regarded as a consequence of a very strong form of the Borel conjecture. In fact, by obstruction theory, it is easy to see that any central element of the fundamental group of an aspherical manifold can be realized by a map of the circle into the space of self homotopy equivalences. See \cite{gott}. What it asked for is that this circle family of self homotopy equivalences be realized by a circle action. This turns out to be too much, and the conjecture is false.

The counterexamples to most of the old conjectures stem from essentially two different constructions of aspherical manifolds. The first was the ``reflection group trick'' of  Michael Davis \cite{davis1} which yielded the first aspherical manifolds whose universal covers are not Euclidean spaces. Also, it was used to construct aspherical manifolds which have no differentiable structure, and even ones which are not triangulable. Moreover, this method produced examples whose fundamental groups are not residually finite, and even ones which have unsolvable word problem. See \cite{davis2} for a survey and references. 

The second construction was Gromov's idea of hyperbolization  \cite{gromov}, developed in \cite{chd, dj, djw, paulin}. It implies that aspherical manifolds exist in abundance. For instance, any compact PL manifold is the image of an aspherical manifold by a degree one tangential map, and any cobordism class can be represented by an aspherical manifold.

In both constructions, the fundamental groups of the aspherical manifolds are centerless. Interestingly, Lee and Raymond \cite{lr2} showed that if the fundamental group of an aspherical manifold has nontrivial center, or more generally contains a nontrivial abelian normal subgroup, then the universal cover is homeomorphic to an Euclidean space. This is rather uncommon in the setting of Davis constructions. 

In this note, motivated by the work of Conner and Raymond, we will give some counterexamples for which the fundamental group has nontrivial center. Ultimately the result is a combination of codimension one splitting obstructions and related surgery methods \cite{ca1, ca2, ca3}, a construction of non-arithmetic hyperbolic manifolds \cite{gp}, a version of hyperbolization \cite{djw}, a construction for the Nielson realization problem \cite{bw}, the work of Borel, Conner and Raymond applying Smith theory to group actions on compact aspherical manifolds \cite{cr1,cr2, cr3, cr4}, and homology manifolds \cite{bfmw} (as the theory in \cite{cr2} does not guarantee that the reduction to a Nielsen type problem will necessarily be on a manifold). We remark that some of these ingredients had been combined in \cite{bw} but here must be modified to deal with homology manifolds; as a technical aside, we point out, following a suggestion of James Davis, how to use \cite{cod} to simplify \cite{bw}. 

To get examples in all sufficiently large dimensions, it is necessary to introduce some of the first steps of a theory of topological $G$-surgery up to ``pseudoequivalence''. See \cite{petrie} for an influential early treatment in the smooth category. There are a number of interesting new issues in the topological category, and this theory will have further applications to transformation groups, but here we develop just enough for the present applications.

\begin{theorem*}
In dimension $6$ and all dimensions $\ge 8$, there are closed aspherical manifolds with fundamental groups each having ${\bb Z}$ as center, yet do not possess nontrivial topological circle actions.
\end{theorem*}

We believe there are examples in all dimensions $\ge 5$ but can only speculate about dimension $4$.  

It is worth noting that counterexamples to the differentiable analogue, just as for the Borel conjecture, are easy consequences of the existence of exotic spheres: for an exotic sphere $\Sigma^n$, the connected sum, $T^n\# \Sigma^n$  does not have any effective smooth circle action.

The examples of the theorem are mapping tori of homeomorphisms of closed, aspherical manifolds, $h\colon V\to V$, such that even though $h^2=h\circ h$ is homotopic to the identity, $h$ is not homotopic to an involution. To get from our initial examples to ones in all dimensions $\ge 8$, we take products of our initial examples with some hyperbolic manifolds. The manifolds $V$ and the homotopy classes of $h$ in the initial examples were constructed for a different application in \cite{bw}, although the fact that the homotopy classes contain homeomorphisms was not noted in that paper. The infinite cyclic center is generated by the square of the stable letter of the HNN-extension describing the fundamental group of the mapping torus.

\medskip

\noindent{\bf Aknowledgement} We would like to express our gratitude to Frank Raymond, who provided valuable comments and suggestions to an earlier draft of the paper.

\section{Construction}

We divide the proof into several small steps, some of which are already known to experts but included for the sake of readability.

\bigskip

\noindent{\bf Step 1}. Equivariant non-rigidity of ${\bb Z}_2$-tori.

\bigskip

Borel conjectured that closed aspherical manifolds are topologically rigid. The equivariant version of this conjecture asserts that closed aspherical manifolds with suitable finite or compact group actions are also rigid. However, the equivariant Borel conjecture is false. The simplest counterexamples are involutions on tori, which we will use in our construction. Very recently, Connolly, Davis and Kahn \cite{codk} gave a very detailed and complete analysis of the equivariant structure set of such involutions.

For closed manifolds in $\dim\ge 5$, or $\dim=4$ and nice fundamental group, the surgery exact sequence
\[
\to L_{n+1}(\pi,\omega) \to S(M) \to H_n(M;{\bb L}) \to L_n(\pi,\omega)
\]
is an exact sequence of abelian groups that computes the $s$-cobordism class $S(M)$ of homology manifolds simple homotopy equivalent to $M$. The surgery obstruction groups $L_n$ are defined in \cite{wall2} and depend on the dimension $n$, the fundamental group $\pi$, and the orientation character $\omega\colon \pi\to\{\pm 1\}$. The homology $H_n$ is the generalized homology theory associated to the (simply connected) surgery spectrum ${\bb L}$. For noncompact manifolds, the exact sequence (where the homology is the ordinary one) computes the structures that have compact support, i.e., for which the homotopy equivalences are homeomorphisms (or CE-maps) outside of some compact subsets. If the surgery obstruction group is replaced by a relative surgery obstruction group that takes into account the fundamental group at infinity, and the homology is the locally finite homology, then the exact sequence computes the proper structures.

We will not be concerned with the usual $s,h$ decorations in the surgery exact sequence which reflect various algebraic $K$-theoretic refinements, as all the discrete groups we are considering will have vanishing Whitehead and reduced projective class groups, so that there is no difference among these decorations.

Let $T$ be the torus $T^n$ with standard involution, such that the fixed set has dimension either $0$ or $1$. In the surgery exact sequence, we take $M$ to be (a compactification of) the complement of the fixed point set of $T$.  Thus $\pi$ is the orbifold fundamental group of the involution, and is a semidirect product ${\bb Z}^n\rtimes {\bb Z}_2$ according to the representation of ${\bb Z}_2$ on $H_1(T)$. This representation is either the multiplication by $-1$ (when the fixed set is discrete) or has a one dimensional trivial summand (when the fixed set has dimension $1$).

The semidirect product $\pi={\bb Z}^n\rtimes {\bb Z}_2$ has ${\bb Z}^{d+1}\rtimes {\bb Z}_2\cong {\bb Z}^d\times({\bb Z}_2 * {\bb Z}_2)$ as a direct summand, where $d$ is the dimension of the fixed set, and ${\bb Z}_2$ acts trivially on the ${\bb Z}^d$ part of ${\bb Z}^{d+1}$. The surgery obstruction group of the direct summand is given as follows.

\begin{theorem*}[Banagl and Ranicki \cite{br}, Connolly and Davis \cite{cod}]
The surgery obstruction group $L_r({\bb Z}_2 * {\bb Z}_2, \omega\oplus\omega)$ is infinitely generated if and only if either $\omega$ is trivial and $r = 2$ or $3$ mod $4$, or $\omega$ nontrivial and $r = 0$ or $1$ mod $4$. Moreover, crossing with a circle produces infinitely generated subgroups in $L_{r+1}({\bb Z}\times({\bb Z}_2 * {\bb Z}_2), \omega\oplus\omega)$ in the same cases. 
\end{theorem*}

The case of $\omega$ trivial and $r = 2$ mod $4$ was due to Cappell \cite{ca1} and expressed there as the geometric statement that there are infinitely many manifolds that are homotopy equivalent to ${\bb R}P^{4k+1}\#{\bb R}P^{4k+1}$, but are not the connected sums of two manifolds homotopy equivalent to ${\bb R}P^{4k+1}$.

The actual result in \cite{cod} is a calculation of the UNil groups of Cappell \cite{ca2} that describe the failure of reduced $L$-theory to be additive for free products (or more generally the Mayer-Vietoris sequences associated to group actions on trees). They subsequently analyzed the complete obstruction to this additivity problem. The groups in the theorem are infinitely generated exactly when the UNil groups are nontrivial. In case the fixed set is discrete, Connolly, Davis and Kahn \cite{codk} further showed that $S(M)$ is a sum of such UNil groups. Moreover, they also showed that $S(M)$ is the same as the isovariant structure set $S^{\text{iso}}(T)$ of exotic involutions on the torus that are isovariantly homotopic to the standard involution. Combining these ingredients with Shaneson's formula for surgery groups after crossing with ${\bb Z}$ \cite{shaneson}, we have a similar description when the fixed set has dimension $1$.

For the present purpose, all we need is that a nontrivial element of $L_{n+1}({\bb Z}^d\times({\bb Z}_2 * {\bb Z}_2), \omega\oplus\omega)$ gives, via split injections ${\bb Z}^d\times({\bb Z}_2 * {\bb Z}_2)\to\pi$, an involution on the torus that is homotopic to the standard involution but not isovariantly homeomorphic to it. We denote the torus with this exotic involution by $T'$. It is important to note that the exotic involution remains PL. By proper surgery, the exotic involutions do not become standard even when the fixed sets are deleted. Thus, there is no need for stratified surgery techniques \cite{weinberger2} for this part of the analysis.

When is $L_{n+1}({\bb Z}^d\times({\bb Z}_2 * {\bb Z}_2), \omega\oplus\omega)$ nontrivial, so that exotic involutions exist? If $d=0$ and $n$ is even, then $\omega$ is trivial, and $L_{n+1}$ is nontrivial only when $n+1=2$ or $3$ mod $4$. Therefore $n=2$ mod $4$. If $d=0$ and $n$ is odd, then $\omega$ is nontrivial, and $L_{n+1}$ is nontrivial only when $n+1=0$ or $1$ mod $4$. Therefore $n=3$ mod $4$. 

\begin{conclusion*}
For $n\ge 4$, there are exotic involutions on $T^n$ with discrete fixed set if and only if $n=2$ or $3$ mod $4$. Similarly, there are exotic involutions on $T^n$ with $1$-dimensional fixed set if and only if $n=0$ or $3$ mod $4$. 
\end{conclusion*}

Relying on \cite{ca1}, Block and Weinberger \cite{bw} made use of these elements, with a simpler analysis in the $d = 0$ case.

\bigskip

\noindent{\bf Step 2}. Construction of counterexamples to Nielsen realization problem.

\bigskip

We largely follow \cite{bw}, but make use of some more surgery to yield examples in all dimensions.  The new difficulties in verifying these examples will be dealt with in the next step.

The torus $T$ with the standard involution is easily the boundary of a PL aspherical manifold $W$ with involution, such that the inclusion $\pi_1T\to \pi_1W$ is injective. We will show that the same is true for the torus $T'$ with exotic involution constructed in the first step.

The construction of $T'$ shows that there is a normal cobordism $Z$ between the standard involution $T$ and the exotic involution $T'$. Moreover, the cobordism is trivial on the fixed part, the involution action is finite and PL, and it can be arranged so that the orbifold fundamental group of $Z/{\mathbb Z}_2$ is also $\pi$. Now we may apply the relative version \cite{djw} of Gromov hyperbolization \cite{chd, dj}, which starts from any simplicial complex and combines ``simplices of non-positively curved manifolds'' to build a polyhedral space that closely resembles the polyhedron \cite{davis2, dj}. So we hyperbolize the quotient $Z/{\bb Z}_2$ relative to the two ends, and get an aspherical ${\bb Z}_2$-cobordism $Z'$ between $T$ and $T'$. Then $W'=Z'\cup_TW$ is a PL aspherical manifold with involution, such that $\pa W'=T'$, and the inclusion $\pi_1T'\to \pi_1W'$ is injective.

Although $\pa W=T$ and $\pa W'=T'$ are not equivariantly homeomorphic, the involutions are homotopic. Any one such homotopy equivalence is homotopic to a homeomorphism. Let $V$ be obtained by glueing $W$ and $W'$ along this homeomorphism. Then $V$ is a closed aspherical manifold of dimension $n+1$. The construction is somewhat reminiscent of the construction in \cite{gp} of non-arithmetic lattices by glueing hyperbolic manifolds with totally geodesic boundary together via isometries of their respective boundaries.

Consider $V$ as $W\cup_T T^n\times[0,1]\cup_{T'} W'$. We have the standard involution on $T^n\times 0$ and the exotic involution on $T^n\times 1$. Since homotopic homeomorphisms of the torus are always pseudoisotopic (see \cite{wall2}, for example), the involutions on both ends of the ribbon $T^n\times[0,1]$ can be extended to a (noninvolutive) homeomorphism of the interior. This gives a homeomorphism $h\colon V\to V$ that satisfies $h^2=id$ on $W$ and $W'$ but is not involutive on $V$. 

Since $h$ is involutive on $W$ and $W'$, we have $h^2\simeq id$ on $V$. However, it is shown in \cite{bw} (for the examples coming from $L_6$, but the proof is no different for all of our examples) that no manifold homotopy equivalent to $V$ has an involution inducing the same outer automorphism as $h_*$ on the fundamental group. In particular, the outer automorphism $h_*$ on the fundamental group of $V$ gives a counterexample to the generalized Nielsen realization problem. We will see in the next step that these remain counterxamples after crossing with any closed hyperbolic manifold.

The closed aspherical manifold with center in its fundamental group that we will ultimately show does not have a circle action is the mapping torus $T(h)$, which has dimension $n+2$. The dimensions  are those $\ge 6$ that are $\neq 3$ mod $4$. To get examples at all dimensions $\ge 8$, we will further consider $T(h)\times H$ for a closed hyperbolic manifold $H$. Note that if the exotic involution is concordant to the standard one, then the extension to the interior of the ribbon can be kept involutive, so that $T(h)$ would have a circle action whose orbits would go through the section $T^n$ (generically) twice. Our aim is to show that $T(h)$ as constructed (and its extension $T(h)\times H$) does not have any circle actions.

\bigskip

\noindent{\bf Step 3}. Verification of counterexamples to Nielsen.

\bigskip

The non-existence of circle actions on $T(h)\times H$ will be reduced to the following version of the counterexample to Nielsen: No homology manifold homotopy equivalent to $V\times H$ has an involution inducing the same outer automorphism on the fundamental group as $h_*$. This is a stronger version than the one proved in \cite{bw}, which only dealt with topological manifolds and did not consider the closed hyperbolic manifold factor.

According to \cite{bfmw}, the surgery theory for topological manifolds used in the argument in \cite{bw} can be adopted for homology manifolds without change, except for perhaps an extra ${\bb Z}$, reflecting, in the normal invariants, Quinn's obstruction to resolution of homology manifolds \cite{quinn2}. In fact, this extra ${\bb Z}$ does not appear here, because a homology manifold homotopy equivalent to $V$ has a Poincar\'e embedded codimension $1$ torus. Therefore, by an argument similar to one in \cite{blw}, the homology manifold must automatically be resolvable, which disposes of this final ${\bb Z}$.

In some more detail, codimension one splitting gives a map from the surgery exact sequence of $V$ to the surgery exact sequence of the codimension one submanifold $T^n\sub V$
\begin{equation*}\begin{CD}
S(V) @>>> H_{n+1}(V;{\bb L}) @>>> L_{n+1}(\pi_1V) \\
@VVV @VVV \\
S(T^n) @>>> H_n(T^n;{\bb L})\otimes {\bb Z}[\frac{1}{2}]
\end{CD}\end{equation*}
The final ${\bb Z}$ is caught by $H_n(T^n;{\bb L})\otimes {\bb Z}[\frac{1}{2}]$ but $S(T^n)$ is trivial.

Notice, though, that all that is used is the vanishing of the homomorphism $S(T^n) \to H_n(T^n;{\bb L})\otimes {\bb Z}[\frac{1}{2}]$, i.e., the Novikov conjecture for the torus. Since the Novikov conjecture is well-known for non-positively curved closed manifolds \cite{fh}, this argument remains valid after crossing with a closed hyperbolic manifold $H$.

Having dealt with the issue of homology manifolds, the remaining obstacle for carrying out the argument in \cite{bw} is that crossing with a hyperbolic manifold may raise the dimension of the fixed sets, so that one does not have a good equivariant homotopy model for the putative group action. In the situation of $T(h)$ without crossing with $H$, the fixed set is low enough dimensional that the local homotopy theory forces the fixed set to be a manifold, and then more standard tools can be applied. When the dimension of the fixed set is higher, the fixed set need not even be ANRs, and have fairly uncontrolled homotopy theory away from $2$. 

Nevertheless, as we observe below, some part of the total surgery obstruction for $T(h)$ does apply a priori to our situation even after crossing with a positive dimensional manifold. To do this, we will begin the development of a theory of surgery obstructions to pseudoequivalence.

\begin{definition*}
A pseudoequivalence $f\colon X\to Y$ between $G$-spaces is an equivariant map which is a homotopy equivalence (upon ignoring the group action). 
\end{definition*}

Let $EG$ be a contractible free $G$-space. The pseudoequivalence simply means that the map $X\times EG\to Y\times EG$ is an equivariant homotopy equivalence. This is also equivalent to that the Borel constructions $(X\times EG)/G\to (Y\times EG)/G$ are homotopy equivalent.

In our situation, we have a $Y$ that is an equivariant Poincar\'e object, and the manifold solution of the Nielsen realization problem would be pseudoequivalent to $Y$. In \cite{rw}, pseudoequivalence invariance of (higher) signature operators is studied. Unfortunately, the analysis there is not adequate for our situation that uses an obstruction at the prime $2$ that does not appear to be an invariant of the any type of signature operator. We shall not use a priori absolute invariants, but have to use relative invariants to capture our obstruction.

\begin{proposition*}
Suppose that $f\colon N\to Y$ is a degree one $G$-map from a manifold to an equivariant Poincar\'e complex that is a pseudoequivalence on the singular sets of these spaces, and that after crossing with $EG$, this map is covered by surgery bundle data. Then it is possible to define a surgery obstruction associated to $f$ that lies in
$L_d^p(\pi_1((Y\times EG)/G))$.
\end{proposition*}

\begin{proof}
We observe that the usual definition of the surgery obstruction actually applies to the chain complex of the mapping cylinder of $f$, once one verifies that it is chain equivalent to a finite projective chain complex \cite{ranicki}. This is easy, since the map is clearly a $G$-equivalence after crossing with any free $G$-space, so one can cross with a highly connected free $G$-manifold (e.g. the universal cover of the boundary of a regular neighborhood of a skeleton of $BG$ embedded in Euclidean space), use the free structure on that map, and then truncate at $\dim Y$. 

This truncated chain complex is projective as shown by Wall \cite{wall1}. In general, for nonfree actions, one cannot improve this to a free chain complex as examples of \cite{quinn1} or \cite{weinberger1} show.
\end{proof}

\begin{remark}
This proposition is in part motivated by an elegant paper of Dovermann \cite{dovermann}, where he observes that in a particular pseudoequivalence problem, the appropriate home for bundle data is the Borel construction.
\end{remark}

\begin{remark}
A variant of this proposition, without bundle data, can be formulated using visible $L$-theory \cite{weiss}.  Even in this setting, it is necessary to work with maps that are pseudoequivalences on the singular sets, because without that condition, there is no way to define the relevant truncation.

Because of Weiss' fundamental result relating assembly maps in visible and quadratic $L$-theories, the subsequent arguments can be adapted to the visible setting. 
\end{remark}

\begin{remark}
Since the surgery obstruction that we have obtained lies in $L^p$, we cannot always complete surgery when this obstruction vanishes. Its vanishing only assures that the surgery can be completed after crossing with ${\bb R}$. We hope to return to the geometry of surgery up to pseudoequivalence in the topological category and its applications to transformation groups in a future paper. For our current purpose, the previous proposition defining the obstruction suffices.
\end{remark}

In light of the proposition, if a stratified Poincar\'e complex is pseudoequivalent to a $G$-manifold, one sees that the top pure stratum, thought of as a pair, has a degree one normal map. The surgery obstruction of this lies in $L_d^p(\pi_1((Y\times EG)/G))$, and its image in $L_d^p(\pi_1((Y\times EG)/G),\pi_1((\Sigma_Y\times EG)/G))$ is trivial ($\Sigma_Y$ denoting the singular set of $Y$), because one obtains a homotopy equivalence after glueing in the neighborhood of $\Sigma_N$.

In our situation, however, the UNil obstruction survives the removal of the singular set. Similar to the argument in \cite{bw}, this is seen by considering the boundary map in the exact sequence of an amalgamated free product that calculates $\pi_1V$ and then mapping further to the pair $({\bb Z}_2*{\bb Z}_2, {\bb Z}_2\sqcup {\bb Z}_2)$.

Now, if we cross with a hyperbolic manifold, then the Novikov conjecture with coefficients in the ring ${\bb Z}[\pi_1V\rtimes {\bb Z}_2]$ implies that our UNil obstruction still survives. The reason is that the obstruction survives upon crossing with ${\bb R}^n$, as an element in the bounded $L$-theory over ${\bb R}^n$ \cite{fp}. Thus we can take the transfer to the universal cover ${\bb R}^n$ of the hyperbolic manifold, followed by the inverse of the exponential map. Actually, this application of bounded $L$-theory is behind the descent proofs of the Novikov conjecture with coefficients; see \cite{cp, fw1, fw2}, for examples.

\bigskip

\noindent{\bf Step 4}. Double cover of nontrivial circle action on $T(h)$.

\bigskip

Since $V$ is aspherical, its fundamental group is torsion free. As observed in \cite{bw}, the fundamental group is also centerless. This is because it is constructed via an amalgamated free product along a subgroup containing no central elements. Indeed, the fundamental group of a hyperbolized cobordism generally contains ${\bb Z}*\pi_1\pa$. 

The fundamental group of the mapping torus $T(h)$ is the HNN-extension
\[
\pi_1(T(h))=\pi_1V\rtimes_{h_*} \langle t\rangle,\quad tat^{-1}=h_*(a).
\]
By $h^2\simeq id$, we know $t^2$ is a central element of $\pi_1(T(h))$. The two fold cover of $T(h)$ corresponding to the subgroup $\pi_1V\times \langle t^2\rangle$ is $T(h^2)$. Moreover, $h^2\simeq id$ implies that $T(h^2)\simeq V\times S^1$. 

Suppose $T(h)$ has a nontrivial circle action. Let
\[
G=\{g\colon T(h^2)\to T(h^2)\colon g\text{ covers the action of some circle element on }T(h)\}
\]
be the group of covering actions. Then we have an exact sequence $1\to {\bb Z}_2\to G\to S^1$.

Since $T(h)$ is a closed aspherical manifold, according to Borel the orbit of any point gives an injective map $\pi_1S^1\to \pi_1(T(h))$, and the image of the map lies in the center of $\pi_1(T(h))$. 

Let $at^k\in \pi_1V\rtimes_{h_*} \langle t\rangle$ be a central element. We have $at^k\cdot b=b\cdot at^k$ for any $b\in \pi_1V$. This means that $h_*^k(b)=a^{-1}ba$. Since $h_*$ is not a conjugation and has prime order $2$, we see that $k$ is even. Then $h_*^k=id$, and $a$ is in the center of $\pi_1V$. Since $\pi_1V$ is centerless, we get $a=1$ and conclude that the center of $\pi_1V\rtimes_{h_*} \langle t\rangle$ is infinite cyclic and generated by $t^2$.

Now the image of the map $\pi_1S^1\to \pi_1(T(h))$ lies in the image of the two fold cover $\pi_1(T(h^2))$. Therefore the circle action lifts to a circle action on $T(h^2)$. This lifting is a splitting to the homomorphism $G\to S^1$ and is compatible with the actions on $T(h)$ and $T(h^2)$. This implies that $G\cong {\bb Z}_2\times S^1$. Since the splitting to the projection $G\to S^1$ is unique, the decomposition $G\cong {\bb Z}_2\times S^1$ is unique.

\bigskip

\noindent{\bf Step 5}. Nontrivial circle action on $T(h^2)$.

\bigskip

The hypothetical action (at the end, the action does not exist) of $G\cong {\bb Z}_2\times S^1$ on $T(h^2)$ induces an action of ${\bb Z}_2$ on the quotient $T(h^2)/S^1$. Although the circle action on $T(h^2)$ may not be unique, we will argue that the composition $V\sub T(h^2)\to T(h^2)/S^1$ is a homotopy equivalence. For this purpose, all we need is that $X=T(h^2)$ is a closed aspherical manifold with fundamental group $\pi\times {\bb Z}$ with centerless $\pi$, and $V\sub X$ is an aspherical submanifold with fundamental group $\pi$.

Since $X$ is a closed aspherical manifold, according to Borel the induced map $\pi_1S^1\to \pi_1X$ is injective and maps into the center ${\bb Z}$ of $\pi_1X=\pi\times {\bb Z}$. By \cite[Lemma 1]{cr4} (also see \cite[Lemma 11.7.2]{lr1}), the quotient ${\bb Z}/\text{image}(\pi_1S^1)$ contains no element of finite order. Therefore $\pi_1S^1$ maps isomorphically onto the subgroup ${\bb Z}$ in $\pi_1X$.

Now we claim that $S^1$ acts freely on $X$. The reason is that if a point $x\in X$ is fixed by ${\bb Z}_k\sub S^1$, then the map $\pi_1S^1\to \pi_1X$ is a composition
\[
\pi_1S^1\to \pi_1(S^1/{\bb Z}_k)\to \pi_1(S^1x)\to \pi_1X.
\]
Since the first map ${\bb Z}\to {\bb Z}$ is multiplication by $k$, the image of $\pi_1S^1$ lies in $k{\bb Z}$ and cannot be the whole subgroup ${\bb Z}$ in $\pi_1X$.

Once we know the action is free, we have a bundle
\[
S^1\to X\to X/S^1.
\]
Then the long exact sequence of homotopy groups implies that 
\[
X/S^1\simeq K(\pi_1X/\text{center}(\pi_1X),1).
\]
Moreover, the composition $V\sub X\to X/S^1$ induces an isomorphism on the fundamental group. Since $V$ is aspherical, this implies that the composition is a homotopy equivalence.

\begin{remark*}
A consequence of the discussion is the homotopy uniqueness of the quotient of the circle action. Since the quotient need not be a manifold (only a homology manifold) and may allow variance by $h$-cobordism, there are many isomorphism classes of  circle actions. The issue of homology manifold can be dealt with by using the resolution theory for homology manifolds to study the $1$-parameter families of group actions. The issue of $h$-cobordism disappears if the Whitehead group of $\pi$ is trivial (which is a very weak part of the Borel conjecture). The quotient is, at least in high dimensions, unique up to these two issues.
\end{remark*}

\bigskip

\noindent{\bf Step 6}. $T(h)\times H$ has no nontrivial circle action.

\bigskip

For an effective $S^1$ action on $T(h)$, Step 4 tells us that it lifts to a free action on the double cover $T(h^2)$ and Step 5 tells us that the composition $V\sub T(h^2)\to T(h^2)/S^1$ is a homotopy equivalence.

The double cover induces a covering involution that takes a fibre $V\sub T(h^2)$ to the fibre $V\sub T(h^2)$ half a circle away via $h$. Step 4 also tells us that this involution commutes with the lifted circle action, and therefore induces a covering involution on the quotient $T(h^2)/S^1$. Then by tracing the projection map, we find that the homotopy equivalence $V\simeq T(h^2)/S^1$ takes the homotopy class of the involution $h$ on $V$ to the homotopy class of the covering involution on $T(h^2)/S^1$.

Borel showed that a finite group of homeomorphisms of a closed aspherical manifold $M$ is represented faithfully in $Out(\pi_1M)$. Although the fact is proved in \cite{cr3} for manifolds, the proof is homological in nature and applies, with no change, to homology manifolds. Now $T(h^2)/S^1$ is a (possibly homological) manifold homotopy equivalent to $V$. Moreover, $T(h^2)/S^1$ has an involution that induces the same automorphism on $\pi_1V$ as $h_*$, which is not an inner automorphism. In Step 3, we showed that this is impossible. 

Finally, we note that the discussion in Steps 4 and 5 is still valid if we multiply with a closed aspherical manifold $H$ with centerless fundamental group, such that the universal cover of $H$ has a degree one Lipschitz map to ${\bb R}^n$ (any hyperbolic manifold, for example). So $T(h)\times H$ also has no nontrivial circle action.


\begin{thebibliography}{1} 

\bibitem{bw}
{\sc J. Block, S. Weinberger}: 
{\em On the generalized Nielsen realization problem}, Comm. Math. Helv. {\bf 83}(2008)21-33

\bibitem{br}
{\sc M. Banagl, A. Ranicki}:
{\em Generalized Arf invariants in algebraic $L$-theory}, Adv. Math. {\bf 199}(2006)542-668

\bibitem{bl}
{\sc A. Bartels, W. L\"{u}ck}:
{\em The Borel conjecture for hyperbolic and CAT(0)-groups}, to appear in Ann. of Math.

\bibitem{blw}
{\sc A. Bartels, W. L\"{u}ck, S. Weinberger}:
{\em On hyperbolic groups with spheres as boundaries},  J. Diff. Geom. {\bf 86}(2010)1-16

\bibitem{borel}
{\sc A. Borel}: 
{\em On periodic maps of certain $K(\pi,1)$}, In: {\OE}uvres: Collected Papers III, 57-60. Springer, 1983

\bibitem{bfmw}
{\sc J. Bryant, S. Ferry, W. Mio, S. Weinberger}: 
{\em Topology of homology manifolds}, Ann. of Math. {\bf 143}(1996)435-467

\bibitem{ca1}
{\sc S. Cappell}: 
{\em On connected sums of manifolds}, Topology {\bf 13}(1974)395-400

\bibitem{ca2}
{\sc S. Cappell}: 
{\em Unitary nilpotent groups and Hermitian $K$-theory I}, Bull. Amer. Math. Soc. {\bf 80}(1974)1117-1122

\bibitem{ca3}
{\sc S. Cappell}: 
{\em A splitting theorem for manifolds}, Invent. Math. {\bf 33}(1976)69-170

\bibitem{cp}
{\sc G. Carlsson, E. Pedersen}: 
{\em Controlled algebra and the Novikov conjectures for $K$- and $L$-theory}, Topology {\bf 34}(1995)731-758

\bibitem{cj}
{\sc A. Casson, D. Jungreis}: 
{\em Convergence groups and Seifert fibered 3-manifolds}, Invent. Math. {\bf 118}(1994)441-456

\bibitem{chd}
{\sc R. Charney, M. Davis}:
{\em Strict hyperbolization}, Topology {\bf 34}(1995)239-350


\bibitem{cod}
{\sc F. Connolly, J. Davis}:
{\em The surgery obstruction groups of the infinite dihedral group}, Geom. Topol. {\bf 8}(2004)1043-1078

\bibitem{codk}
{\sc F. Connolly, J. Davis, Q. Kahn}:
{\em Topological rigidity and $H_1$-negative involutions on tori},
(preprint)

\bibitem{cr1}
{\sc P. Conner, F. Raymond}:
{\em Actions of compact Lie groups on aspherical manifolds}, In: Topology of Manifolds, University of Georgia, Athens, Georgia 1969, 227-264. Markham, Chicago 1970 

\bibitem{cr2}
{\sc P. Conner, F. Raymond}:
{\em Injective operations of the toral groups}, Topology {\bf 10}(1971)283-296

\bibitem{cr3}
{\sc P. Conner, F. Raymond}:
{\em Manifolds with few periodic homeomorphisms}, In: Proceedings of the Second Conference on Compact Transformation Groups, University of Massachusetts, Amherst 1971, 1-75, Lecture Notes in Math. {\bf 299}. Springer, 1972

\bibitem{cr4}
{\sc P. Conner, F. Raymond}:
{\em Realizing finite groups of homeomorphism from homotopy classes of self-homotopy equivalences}, In: Manifolds-Tokyo, Proc. Int. Conf., Tokyo 1973, 231-238. Univ. Tokyo Press, Tokyo 1970 

\bibitem{davis1}
{\sc M. Davis}:
{\em Groups generated by reflections and aspherical manifolds not covered by Euclidean space}, Ann. of Math. {\bf 117}(1983)293-324

\bibitem{davis2}
{\sc M. Davis}:
{\em Exotic aspherical manifolds}, In: Topology of High-dimensional Manifolds, 371-404, ITCP Lect. Notes {\bf 9}, Abdus Salam Int. Cent. Theo. Phys., Trieste, 2002

\bibitem{dj}
{\sc M. Davis, T. Januszkiewicz}:
{\em Hyperbolization of polyhedra}, J. Diff. Geom. {\bf 34}(1999)347-388

\bibitem{djw}
{\sc M. Davis, T. Januszkiewicz, S. Weinberger}:
{\em Relative hyperbolization and aspherical bordisms}, J. Diff. Geom. {\bf 58}(2001)535-541

\bibitem{docarmo}
{\sc M. do Carmo}:
Riemannian geometry. Birkh\"auser, Boston, 1992

\bibitem{dovermann}
{\sc K. H. Dovermann}:
{\em Diffeomorphism classification of finite group actions on disks}, Topo. and its Appl. {\bf 16}(1983)123-133

\bibitem{farrell}
{\sc F. T. Farrell}:
{\em The Borel conjecture}, In: Topology of High-dimensional Manifolds, 225-298, ITCP Lect. Notes {\bf 9}, Abdus Salam Int. Cent. Theo. Phys., Trieste, 2002

\bibitem{fh}
{\sc F. T. Farrell, W. C. Hsiang}:
{\em On Novikov's conjecture for nonpositively curved manifolds I}, Ann. of Math. {\bf 113}(1981)199-209

\bibitem{fp}
{\sc S. Ferry, E. Pedersen}:
{\em Epsilon surgery theory}, In: Novikov Conjectures, Index Theorems and Rigidity, vol.2 (Oberwolfach, 1993), 167-226,  London Math. Soc. Lecture Note Ser. {\bf 227}. Cambridge Univ. Press, 1995

\bibitem{fw1}
{\sc S. Ferry, S. Weinberger}:
{\em Curvature, tangentiality, and controlled topology}, Invent. Math. {\bf 105}(1991)401-414

\bibitem{fw2}
{\sc S. Ferry, S. Weinberger}:
{\em A coarse approach to the Novikov conjecture}, In: Novikov Conjectures, Index Theorems and Rigidity, vol.1 (Oberwolfach, 1993), 167-226,  London Math. Soc. Lecture Note Ser. {\bf 226}. Cambridge Univ. Press, 1995

\bibitem{gabai}
{\sc D. Gabai}:
{\em Convergence groups are Fuchsian groups}, Ann. of Math. {\bf 58}(2001)535-541

\bibitem{gott}
{\sc D. H. Gottlieb},
{\em A certain subgroup of the fundamental group}, Amer. J. Math. {\bf 87}(1965)840-856


\bibitem{gromov}
{\sc M. Gromov}:
{\em Hyperbolic groups}, In: Essays in Group Theory, 75-263, Math. Sci. Res. Inst. Publ. {\bf 8}. Springer, 1987

\bibitem{gp}
{\sc M. Gromov, I. Piatetski-Shapiro}:
{\em Nonarithmetic groups in Lobachevsky spaces}, Inst. Hautes \'{E}tudes Sci. Publ. Math. {\bf 66}(1988)93-103.


\bibitem{lr1}
{\sc K. B. Lee, F. Raymond}:
Seifert Fiberings. Math. Surv. and Mono. {\bf 166}, AMS, 2010.

\bibitem{lr2}
{\sc R. Lee, F. Raymond}:
{\em Manifolds covered by Euclidean space}, Topology {\bf 14}(1975)49-57

\bibitem{paulin}
{\sc F. Paulin}:
{\em Constructions of hyperbolic groups via hyperbolizations of polyhedra}, In: Group Theory from a Geometrical Viewpoint, Trieste, 1990, 313-372. World Scientific, 1991


\bibitem{petrie}
{\sc T. Petrie}:
{\em Pseudoequivalences of $G$-manifolds}, Proc. Symp. Pure. Math. {\bf 32}(1978)169-210

\bibitem{preissmann}
{\sc A. Preissmann}:
{\em Quelques propri\'et\'es globales des espaces de Riemann}, Comment. Math. Helv. {\bf 15}(1943)175-216


\bibitem{quinn1}
{\sc F. Quinn}:
{\em Ends of maps II}, Invent. Math. {\bf 68}(1982)353-424

\bibitem{quinn2}
{\sc F. Quinn}:
{\em Resolutions of homology manifolds, and the topological characterization of manifolds}. Invent. Math. {\bf 72}(1983)267-284


\bibitem{ranicki}
{\sc A. Ranicki}:
{\em The algebraic theory of surgery I, II}, Proc. London Math. Soc. {\bf 40}(1980)87-283


\bibitem{rw}
{\sc J. Rosenberg, S. Weinberger}:
{\em An equivariant Novikov conjecture}, with an appendix by J. P. May, K-Theory {\bf 4}(1990)29Ð53

\bibitem{shaneson}
{\sc J. Shaneson}:
{\em Wall's surgery obstuction groups for ${\bb Z}\times G$}, Ann. of Math. {\bf 90}(1969)296-334

\bibitem{wall1}
{\sc C. T. C. Wall}:
{\em Finiteness conditions for CW-complexes}, Ann. of Math. {\bf 81}(1965)56-69

\bibitem{wall2}
{\sc C. T. C. Wall}:
Surgery on Compact Manifolds. Academic Press, 1971.

\bibitem{weinberger1}
{\sc S. Weinberger}:
{\em Constructions of group actions}, In: Group Actions on Manifolds, A Survey of Some Recent Developments (Boulder, 1983), 269-298, Contemp. Math. {\bf 36}. AMS, 1985

\bibitem{weinberger2}
{\sc S. Weinberger}:
The Topological Classification of Stratified Spaces. University of Chicago Press,  1993
 
\bibitem{weiss}
{\sc M. Weiss}:
{\em Visible $L$-theory}, Forum Math. {\bf 4}(1992)465-498

\end{thebibliography}
\end{document}